\documentclass[journal]{IEEEtran}
\usepackage{amsmath,amssymb, mathrsfs}
\usepackage{bm}
\usepackage[dvipdfmx]{graphicx}
\usepackage[dvipdfmx]{color}
\usepackage{ascmac}
\usepackage{amsthm}
\usepackage{tikz}
\usepackage{comment}
\usepackage{mathtools}
\usepackage{color}
\usetikzlibrary{calc}
\usetikzlibrary{positioning}
\usetikzlibrary{arrows.meta}

\theoremstyle{definition}
\newtheorem{dfn}{Definition}
\newtheorem{example}[dfn]{Example}
\newtheorem{assum}[dfn]{Assumption}
\newtheorem{thm}[dfn]{Theorem}

\newtheorem{cor}[dfn]{Corollary}

\newtheorem{prob}[dfn]{Problem}
\newtheorem{rem}[dfn]{Remark}

\newcommand{\bE} { {\mathbb E}}
\newcommand{\bR} { {\mathbb R}}

\newcommand{\bP} { {\mathbb P}}
\newcommand{\bZ} { {\mathbb Z}}

\newcommand{\cB} { {\mathcal B}}

\newcommand{\cM} { {\mathcal M}}
\newcommand{\cN} { {\mathcal N}}
\newcommand{\cO} { {\mathcal O}}
\newcommand{\cQ} { {\mathsf Q}}
\newcommand{\cR} { {\mathsf R}}
\newcommand{\cS} { {\mathcal S}}

\newcommand{\cF} { {\mathcal F}}

\newcommand{\ito}[1]{{\color{black}{#1}}}

\newcommand{\sugiura}[1]{{\color{black}{#1}}}

\newcommand{\red}[1]{{\color{black}{#1}}}

\DeclareGraphicsExtensions{.eps, .png, .pdf}
\ifCLASSINFOpdf
\else
\fi
\begin{document}
%
\title{Bayesian Differential Privacy for Linear Dynamical Systems
}
%
%
%

\author{Genki~Sugiura,
        Kaito~Ito,~\IEEEmembership{Student Member,~IEEE}
        and~Kenji~Kashima,~\IEEEmembership{Senior Member,~IEEE}
\thanks{The authors are with the Graduate School of Informatics, Kyoto University, Kyoto, Japan e-mail:  kk@i.kyoto-u.ac.jp.}
\thanks{$ \copyright $ 2021 IEEE. Personal use of this material is permitted. Permission from IEEE must be obtained for all other uses, in any current or future media, including reprinting/republishing this material for advertising or promotional purposes, creating new collective works, for resale or redistribution to servers or lists, or reuse of any copyrighted component of this work in other works.}}

%
%

\markboth{}%
{}
%



\maketitle
\thispagestyle{empty}

\begin{abstract}
Differential privacy is a privacy measure based on the difficulty of discriminating between similar input data. 
In differential privacy analysis, \emph{similar} data usually implies that their distance does not exceed a predetermined threshold. 
It, consequently, does not take into account the difficulty of distinguishing data sets that are far apart, which often contain highly private information. 
This problem has been pointed out in the research on differential privacy for static data, and Bayesian differential privacy has been proposed, which provides a privacy protection level even for outlier data by utilizing the prior distribution of the data.
In this study, we introduce this Bayesian differential privacy to dynamical systems, and provide privacy guarantees for distant input data pairs and reveal its fundamental property. For example, we design a mechanism that satisfies the desired level of privacy protection, which characterizes the trade-off between privacy and information utility.
\end{abstract}


%
\IEEEpeerreviewmaketitle

\section{Introduction}\label{sec:introduction}
As the Internet-of-Things (IoT) and cloud computing are attracting more and more attention for their convenience, 
privacy protection and security have become key technologies in control systems. 
To cope with privacy threats, many privacy protection methods have been studied so far~\cite{sweeneyKanonymityModelProtecting2002,machanavajjhalaLdiversityPrivacyKanonymity2006,liTClosenessPrivacyKAnonymity2007}.
Among them, differential privacy \cite{dworkCalibratingNoiseSensitivity2006} has been used to solve many privacy-related problems in areas such as smart grid \cite{sandbergDifferentiallyPrivateState2015}, health management \cite{dankarApplicationDifferentialPrivacy2012}, and blockchain \cite{yangDifferentiallyPrivateData2018}, because it can mathematically quantify privacy guarantees.
Differential privacy has originally been applied to static data, but as shown in the example of power systems above, there is an urgent need to establish privacy protection techniques for dynamical systems.
In recent years, the concept of differential privacy has been introduced to dynamical systems \cite{lenyDifferentiallyPrivateFiltering2014a}, and from the viewpoint of control systems theory, the relationship between privacy protection and the observability of systems has been clarified, and methods of controller design with privacy protection in mind have been studied \sugiura{\cite{kawanoDesignPrivacyPreservingDynamic2020, cortes_differential_2016, katewa_privacy_2018}}.

Conventional differential privacy is a privacy measure based on the difficulty of distinguishing similar data, where $x$ and $x'$ are regarded as being similar if $|x-x'|\le c$ for a prescribed $c>0$. To put it conversely, there is no indistinguishability guarantee for $x$ and $x'$ if $|x-x'|> c$. 
This implies there is a risk of information leakage when there are outliers from normal data as pointed out in~\cite{Ito2021heavy}.
For example, unusual electricity consumption patterns may contain highly private information about the lifestyle. 
In the literature \cite{triastcynBayesianDifferentialPrivacy2020b}, a new concept called Bayesian Differential Privacy is developed for static data to solve this problem.
Bayesian differential privacy considers the underlying probability distribution of the data, and attempts to guarantee privacy for data sets that are far apart.

In this study, we consider a prior distribution for the signal that we want to keep secret and introduce Bayesian differential privacy for linear dynamical systems. 
Similar to the conventional differential privacy cases \cite{kawanoDesignPrivacyPreservingDynamic2020}, we consider a mechanism where stochastic noise is added to the output data.
Note that applying a large noise increases the privacy protection level, but decreases the information usefulness
\cite{kawanoModularControlPrivacytobepublished}.
In Theorem \ref{thm:BDP_dynamic} below, a lower bound of noise scale to guarantee a prescribed Bayesian differential privacy level will be derived. Other properties including the relation to the conventional case are investigated based on this. 
The rest of this paper is organized as follows.
In Section~\ref{sec:works_and_setup}, we introduce differential privacy for dynamical systems.
In Section~\ref{sec:with_prior}, we propose Bayesian differential privacy for dynamical systems and derive a sufficient condition for added noise to achieve its privacy guarantee.
In Section~\ref{sec:modeling_mechanism}, considering the trade-off between privacy and information utility, we derive the Gaussian noise with the minimum energy while guaranteeing the Bayesian differential privacy.
In Section~\ref{sec:example}, the usefulness of Bayesian differential privacy is described via a numerical example.
Some concluding remarks are given in Section~\ref{sec:conclusion}.

\paragraph*{Notations}
The sets of real numbers and nonnegative integers are denoted by $\bR$ and $\bZ_+$, respectively.
The imaginary unit is denoted by $ {\rm j} $.
For vectors $x_1,\dots, x_m \in \bR^n$, a collective vector $[x_1^\top \cdots x_m^\top]^\top \in \bR^{nm}$ is described by $[x_1;\cdots;x_m]$ for the sake of simplicity of description.
For a sequence $u(t) \in \bR^n,\ t = 0, 1, \ldots, T$, a collective vector is denoted by $U_T \coloneqq [u(0); \cdots ;u(T)] \in \bR^{(T+1)n}$ using a capital alphabet.
For a square matrix $A \in \bR^{n \times n}$, its determinant is denoted by $\mathrm{det}(A)$, and when its eigenvalues are real, its maximum and minimum eigenvalues are denoted by $\lambda_{\max}(A)$ and $\lambda_{\min}(A)$, respectively.
We write $ A\succ 0 $ (resp. $ A\succeq 0 $) if $ A $ is positive definite (resp. semidefinite).
For $A\succeq 0$, the principal square root of $A$ is denoted by $A^{1/2}$.
The identity matrix of size $n$ is denoted by $I_n$. The subscript $n$ is omitted when it is clear from the context. 
\red{The Euclidean norm of a vector $x\in\bR^n$ is denoted by $| x |$, and 
	its weighted norm with $A \succ 0$ is denoted by $|x|_A \coloneqq (x^\top A x)^{1/2}$.}
The indicator function of a set $ S\subset \bR^n $ is denoted by $ 1_S $, i.e., $ 1_S (x) = 1 $ if $ x\in S $, and $ 0 $, otherwise.
For a topological space $X$, the Borel algebra on $X$ is denoted by $\cB(X)$.
Fix some complete probability space $(\Omega, \cF, \bP)$, and let $ \bE $ be the expectation with respect to $ \bP $.
\red{For an $ \mathbb{R}^n $-valued random vector $ w $, $ w \sim  \cN_n (\mu,\Sigma)$ means that $ w $ has a nondegenerate multivariate Gaussian distribution with mean $ \mu \in \bR^n $ and covariance matrix $ \Sigma \succ 0 $.}
The so called $\cQ$-function is defined by $\cQ (c) \coloneqq \frac{1}{\sqrt{2\pi}}\int_c^{\infty} {\rm e}^{-\frac{v^2}{2}}dv$, where $\cQ(c) <1/2$ for $c>0$, and $\cR (\varepsilon,\delta)\coloneqq(\cQ^{-1}(\delta) + \sqrt{(\cQ^{-1}(\delta))^2 + 2 \varepsilon})/2\varepsilon$.
The gamma function is \red{denoted by $ \Gamma(\cdot) $}.
A random variable $z$ is said to have a \ito{$ \chi^2 $} distribution with $k$ degrees of freedom, denoted by $z \sim \chi^2_k$, if its distribution has the following probability density:
\begin{align*}
p(z;k) = \frac{\ito{z}^{k/2-1}\mathrm{e}^{-\sugiura{z}/2}}{\,2^{k/2} \Gamma(k/2)},\quad z \ge 0,\ k \in \{1, 2, \ldots\}.
\end{align*}

\section{Conventional differential Privacy for dynamical systems}
\label{sec:works_and_setup}
In this section, we briefly overview fundamental results on differential privacy for dynamical systems.
Consider the following discrete-time linear system:
\begin{align}
\left\{\begin{array}{l}
x(t+1) = A x(t) + B u(t), \\
y(t) = C x(t) + D u(t),
\end{array}\right. \label{sys}
\end{align}
for $t \in \bZ_+$, where $x(t) \in \bR^n$, $u (t) \in \bR^m$, and $y(t) \in \bR^q$ denote
the state, input, and output, respectively, and $A \in \bR^{n\times n}$,
$B \in \bR^{n\times m}$, $C \in \bR^{q\times n}$, and $D \in \bR^{q\times m}$.
For simplicity, we assume $x(0)=0$ and 
the information to be kept secret is the input sequence $U_T$ up to a finite time $T \in \bZ_+$. 
For \eqref{sys}, the output sequence $Y_T \in \bR^{(T+1)q}$ is described by
\begin{align}
Y_T = N_T U_T, \label{Outseq}
\end{align}
where $N_T \in \bR^{(T+1)q \times (T+1)m}$ is
\begin{align}
N_T & = {\cal T}_T(A,B,C,D)\\
& \coloneqq\left[\begin{array}{cccccc}
D & 0 & \cdots & \cdots &  0\\
CB & D & \ddots &  & \vdots \\
CAB & CB & D &\ddots & \vdots \\
\vdots & \vdots & \ddots & \ddots &0 \\
CA^{T-1}B & CA^{T-2}B & \cdots & CB & D
\end{array}\right].
\label{Markov}
\end{align}
To proceed with differential privacy analysis, we consider the output $y_w(t)\coloneqq y(t) + w(t)$ after adding the noise $w(t) \in \bR^q$; see Fig.~\ref{fig:diagram_out_noise}.
From \eqref{Outseq}, $Y_{w,T} \in \bR^{(T+1)q}$ can be described by
\begin{align}
Y_{w,T} = N_T U_T + W_T, \label{Out_noise}
\end{align}
which defines a mapping $\cM \colon \bR^{(T+1)m} \times \Omega \ni (U_T, \omega) \mapsto Y_{w,T} \in \bR^{(T+1)q}.$
\begin{figure}[!t]
	\centering
	\includegraphics[height=0.6in]{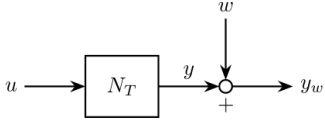}
	\caption{Mechanism with output noise.}
	\label{fig:diagram_out_noise}
\end{figure}
In differential privacy analysis, this mapping is called a \emph{mechanism}.

Next, the definition of the differential privacy is given. We begin with the definition of the data similarity:
\begin{dfn}
	Given a positive definite matrix $K \in \bR^{(T+1)m \times (T+1)m}$, a pair of input data $(U_T, U'_T) \allowbreak \in \bR^{(T+1)m} \times \bR^{(T+1)m}$ is said to belong to the binary relation $K$-adjacency if
	\begin{align}
	|U_T - U'_T|_K \le 1. \label{cK_Adj}
	\end{align}
	The set of all pairs of the input data that are $K$-adjacent is denoted by $\mathrm{Adj}_K$.
\end{dfn}
This $K$-adjacency is an extension of $c$-adjacency, which corresponds to $K=I/c^2$, for the $2$-norm in previous work~\cite{kawanoDesignPrivacyPreservingDynamic2020}. 
Next, we describe the definition of $(K, \varepsilon, \delta)$-differential privacy in dynamical systems in the same way as for static data \cite[Definition 2.4]{kawanoDesignPrivacyPreservingDynamic2020}.
\begin{dfn}[$(K, \varepsilon, \delta)$-differential privacy]\label{def:DP_dynamic_cadj}
	Given $\varepsilon>0$ and $\delta \ge 0$, the mechanism \eqref{Out_noise} is said to be \emph{$(K, \varepsilon, \delta)$-differentially private} ($(K,\varepsilon,\delta)$-DP) at a finite time instant $T \in \bZ_+$, if 
	\begin{align}
	&\bP [N_T U_T + W_T \in \cS] \nonumber \\
	&\le {\rm e}^{\varepsilon} \bP [N_T U'_T + W_T \in \cS] + \delta, \quad \forall \cS \in \cB\left(\bR^{(T + 1)q}\right) \label{diffprv}
	\end{align}
	for any $(U_T,U'_T) \in {\rm Adj}_K$.
\end{dfn}

Suppose that the output sequence $Y_{w,T}$ and the state equation \eqref{sys} are available for an attacker  trying to estimate the value of the input sequence $U_T$. Differential privacy requires the output sequence statistics are close enough at least for adjacent data pairs. 
A sufficient condition for the mechanism induced by Gaussian noise to be $(\varepsilon, \delta)$-differentially private under $c$-adjacency is derived in \cite[Theorem 2.6]{kawanoDesignPrivacyPreservingDynamic2020}. This result can be straightforwardly extended as follows:
\begin{thm}\label{thm:DP_Kadj}
	The Gaussian mechanism \eqref{Out_noise} induced by $W_T \sim \cN_{(T+1)q}(\mu_w,\Sigma_w)$ is $(K, \varepsilon,\delta)$-differentially private at a finite time $T \in \bZ_+$ with $\varepsilon > 0$ and $1/2 > \delta > 0$, if the covariance matrix $\Sigma_w \succ 0$ is chosen such that
	\begin{align}
	\lambda_{\max}^{-1/2}\left(\cO_{K, \Sigma_w, T}\right) \ge \cR (\varepsilon,\delta),\label{ep_ck}
	\end{align}
	where
	\begin{align}
	\cO_{K, \Sigma_w, T} \coloneqq K^{-1/2}N_T^\top \Sigma_w^{-1} N_T K^{-1/2}.
	\end{align}
\end{thm}
\textcolor{black}{Larger noise $ W_T $ and lower threshold for the adjacency in the sense of $ \Sigma $ and $ K $, respectively, make the left-hand side of \eqref{ep_ck} larger. This implies that differential privacy is ensured for smaller $ \varepsilon $ and $ \delta $ because $ \cR$ is a decreasing function.
	Further insight of $ K $ will be revealed in Corollary~\ref{cor:comparison} below.}

\section{Bayesian differential privacy for dynamical systems}\label{sec:with_prior}

\subsection{Formulation}
In Definition \ref{def:DP_dynamic_cadj}, the difficulty of distinguishing data pairs whose $K$-weighted distance is larger than a threshold $1$ is not taken into account. 
Note that there is no design guideline for $K$. 
In this section, we introduce Bayesian differential privacy for dynamical systems. 
To this end, we assume the following availability of the prior distribution of the data to be protected, and provide a privacy guarantee that takes into account the discrimination difficulty of data pairs based on the prior. 

\begin{assum}\label{assum:input_with_prior}
	The input data to the mechanism, $U_T$, is an $\bR^{(T+1)m}$-valued random variable \ito{with distribution $ \bP_{U_T} $. \ito{In addition, one can use the prior $ \bP_{U_T} $ to design a mechanism}}.
\end{assum}

\ito{The following is a typical example where a private input signal is a realization of a random variable.}
\begin{example}\label{example:freq_model}
	Suppose that the input data $u(t)$ to be protected is the reference for tracking control; see  \cite{kawanoModularControlPrivacytobepublished}. In many applications, tracking to the reference signal over specified frequency ranges is required. Such a control objective can be represented by filtering white noise $\xi(t)$. To be more precise, we assume $u(t)=r(t)$ is generated by 
	\begin{align}
	&x_r(t+1) = A_r x_r(t) + B_r \xi(t),\ x_r(0) = 0,  \label{eq:noise_filtering}\\
	& r(t) = C_r x_r(t) + D_r \xi(t), \label{eq:noise_filtering2}\\
	&\xi(t) \sim \cN(0, I),\quad t \in \{0, 1, \ldots, T\},
	\end{align}
	where $A_r \in \bR^{l \times l},\ B_r \in \bR^{l \times m},\ C_r \in \bR^{m \times l},\ D_r \in \bR^{m \times m}$.
	The power spectrum of $u(t)$ is characterized by the frequency transfer function 
	\begin{align}
	\label{eq:TF}
	G_r({\rm e}^{{\rm j}\lambda}):=C_r({\rm e}^{{\rm j}\lambda}I-A_r)^{-1}B_r + D_r,\ \lambda\in [0,\pi).
	\end{align}
	In this case,
	\begin{align}
	r(t) = \begin{cases}
	D_r \xi(0) & (t = 0),\\
	D_r \xi(t) + \sum_{j=1}^{t} C_r A_r^{j-1}B_r \xi(t-j) & (t \ge 1),
	\end{cases}
	\end{align}
	and the prior distribution is given by $U_T \sim {\cal N}(0,\Sigma_U)$ with
	\begin{align}
	& \Sigma_U := \Xi_T \Xi_T^\top, \label{eq:freq_covar}\\
	& \Xi_T := {\cal T}_T( A_r, B_r, C_r, D_r ). \label{eq:freq_covar2}
	\end{align}
	\hfill\qedsymbol
\end{example}

\red{Note that the step reference signal whose value obeys $r(t) \equiv \bar{r} \sim \mathcal{N}(0,\Sigma_{\rm s})$ can be modeled by setting $ A_r  = C_r = I, \ B_r = D_r = 0  $ with $ x_r (0) \sim \mathcal{N} (0,\Sigma_{\rm s}) $ in \eqref{eq:noise_filtering}, \eqref{eq:noise_filtering2}.
	This corresponds to the case where the initial state is the private information rather than the input sequence $ U_T $ as discussed in \cite{kawanoDesignPrivacyPreservingDynamic2020}.}

Then, based on the Bayesian differential privacy for static data~\cite{triastcynBayesianDifferentialPrivacy2020b}, we define $(\bP_{U_T}, \gamma, \varepsilon, \delta)$-Bayesian differential privacy, which is an extension of differential privacy for dynamical systems.
\begin{dfn}[$(\bP_{U_T}, \gamma, \varepsilon, \delta)$-Bayesian differential privacy]\label{def:BDP_dynamic}
	Assume that the random variables $U_T, U'_T$ are independent and both follow the distribution $\bP_{U_T}$.
	Given $1\ge \gamma \ge 0,\ \varepsilon >0$ and $\delta \ge 0$, the mechanism \eqref{Out_noise} is said to be \emph{$(\bP_{U_T}, \gamma, \varepsilon, \delta)$-Bayesian 
		differentially private} ($(\bP_{U_T}, \gamma, \varepsilon, \delta)$-BDP) at a finite time instant $T \in \bZ_+$, if 
	\begin{align}
	&\bP \Bigl[ \bP [N_T U_T + W_T \in \cS \mid U_T] \nonumber \\
	&\le \mathrm{e}^{\varepsilon} \bP [N_T U'_T + W_T \in \cS \mid U'_T] + \delta \Bigr] \ge \gamma, \quad \forall \cS \in \cB(\bR^{(T+1)q}). \label{diffprv_2}
	\end{align}
\end{dfn}
\red{In \eqref{diffprv_2}, the outer (resp. inner) $\mathbb{P}$ is taken with respect to $(U_T, U'_T)$ (resp. $W_T$).}
Roughly speaking, the definition of BDP is that the probability that the mechanism satisfies $(K, \varepsilon, \delta)$-DP is greater than or equal to $\gamma$. Note that this definition places no direct restriction on the distance between a pair of input data $U_T, U'_T$.
\subsection{Sufficient condition for noise scale}

It is desirable that the added noise $w$ is small to retain the data usefulness; see e.g., Section \ref{sec:example}.  
The following theorem gives a sufficient condition for noise scale to guarantee $(\bP_{U_T}, \gamma, \varepsilon, \delta)$-Bayesian differential privacy.
\begin{thm}\label{thm:BDP_dynamic}
	Suppose that the prior distribution $\bP_{U_T}$ of $U_T$ is $\cN_{(T+1)m}(0,\Sigma)$.
	The Gaussian mechanism \eqref{Out_noise} induced by $W_T \sim \cN_{(T+1)q}(\mu_w,\Sigma_w)$ is $(\bP_{U_T}, \gamma, \varepsilon, \delta)$-Bayesian differentially private at a finite time $T \in \bZ_+$ with $1\ge \gamma \ge 0$, $\varepsilon > 0$, and $1/2 > \delta > 0$, if the covariance matrix $\Sigma_w \succ 0$ is chosen such that
	\begin{align}
	\lambda_{\max}^{-1/2}(\cO_{\Sigma, \Sigma_w, T}) \ge c(\gamma,T)\cR(\varepsilon, \delta) \label{eq:cond_BDP}
	\end{align}
	where $\cO_{\Sigma, \Sigma_w, T}$ is defined by
	\begin{align}
	\cO_{\Sigma, \Sigma_w, T} &\coloneqq \Sigma^{1/2}N_T^\top\Sigma_w^{-1}N_T\Sigma^{1/2}
	\end{align}
	and $c(\gamma,T)$ is the unique $c>0$ that satisfies
	\begin{align}
	\gamma &= \int_0^{c^2/2} \dfrac{1}{2^{\sugiura{(T+1)m/2}} \Gamma(\sugiura{(T+1)m/2)}} x^{\sugiura{(T+1)m/2} - 1}\mathrm{e}^{-x/2} dx. \label{eq:chi_squared}
	\end{align}
\end{thm}
\begin{proof}
	Using a similar argument as in the proof for \cite[Theorem 2.6]{kawanoDesignPrivacyPreservingDynamic2020}, for any fixed $U,U' \in \bR^{(T+1)m}$, one has
	\begin{align*}
	& \bP [N_T U_T + W_T \in \cS \mid U_T = U] \\
	&\le {\rm e}^{\varepsilon} \bP [N_T U'_T + W_T \in \cS \mid U'_T = U']  \\
	&+ \bP\left[Z \ge \sugiura{\varepsilon h - \frac{1}{2h}} \mid U_T = U,\ U'_T = U'\right]
	\end{align*}
	where
	\begin{align*}
	h\coloneqq |Y - Y'|_{\Sigma^{-1}_w}^{-1}, \ Y \coloneqq N_T U_T, \ Y' \coloneqq N_T U'_T,
	\end{align*}
	and $Z \sim \cN(0, 1)$. Then, the mechanism is $(\bP_{U_T}, \gamma, \varepsilon, \delta)$-Bayesian differentially private, if $\cQ(\sugiura{\varepsilon h - \frac{1}{2h}}) \le \delta$ with probability at least $\gamma$, i.e.,
	\begin{align}
	\bP[h \ge \cR(\varepsilon, \delta)] \ge \gamma. \label{eq:eval_res}
	\end{align}
	The inequality \eqref{eq:eval_res} holds if \eqref{eq:cond_BDP} is satisfied.
	This is because
	\begin{align*}
	h^{-2} &= |N_T(U_T - U'_T)|^2_{\Sigma^{-1}_w}\\
	&= (U_T - U'_T)^\top \Sigma^{-1/2} \cO_{\Sigma, \Sigma_w, T} \Sigma^{-1/2}(U_T - U'_T) \\
	&\le |\Sigma^{-1/2}(U_T - U'_T)|^2 \lambda_{\max}(\cO_{\Sigma, \Sigma_w, T}),
	\end{align*}
	and then, \sugiura{from the fact that $|\Sigma^{-1/2}(U_T - U'_T)|^2/2$ follows $\chi^2$ distribution with $(T+1)m$ degrees of freedom and \ito{the definition of $ c(\gamma,T) $,}} $|\Sigma^{-1/2}(U_T - U'_T)|^2 \le c(\gamma, T)^2$ with probability $\gamma$.
\end{proof}

In order to clarify the connection between conventional and Bayesian DP, it is worthwhile comparing Theorems \ref{thm:DP_Kadj} and \ref{thm:BDP_dynamic}. 
Bayesian differential privacy with the prior distribution $\bP_{U_T} \sim \cN_{(T+1)m}(0,\Sigma)$ corresponds to  differential privacy with an adjacency weight $K$. 
\begin{cor}
	\label{cor:comparison}
	Suppose that the prior distribution $\bP_{U_T}$ of $U_T$ is $\cN_{(T+1)m}(0,\Sigma)$.
	Let a finite time $T \in \bZ_+$ with $1 \ge \gamma \ge 0$, $\varepsilon > 0$, and $1/2 > \delta > 0$ be given. Then, the Gaussian mechanism \eqref{Out_noise} induced by $W_T \sim \cN_{(T+1)q}(\mu_w,\Sigma_w)$ is $(\bP_{U_T}, \gamma, \varepsilon, \delta)$-Bayesian differentially private at the time $T$, if the mechanism is $(K,\varepsilon,\delta)$-differentially private at the time $T$ with
	\begin{align}
	K:=\Sigma^{-1}/c(\gamma, T)^2 \label{eqn:K_and_Sigma}
	\end{align}
	with $c(\gamma,T)$ defined in Theorem \ref{thm:BDP_dynamic}. 
\end{cor}
\begin{proof}
	Let us define $Y_{w, T} := N_T U_T + W_T,\ Y'_{w, T} := N_T U'_T + W_T.$
	If the mechanism is $(K, \varepsilon,\delta)$-differentially private with $K$ defined in \eqref{eqn:K_and_Sigma},
	\begin{align*}
	&\bP \Bigl[ \bP [Y_{w,T} \in \cS \mid U_T] \le \mathrm{e}^{\varepsilon} \bP [Y'_{w, T} \in \cS \mid U'_T] + \delta \Bigr]\\
	&\ge \bP \Bigl[ \bP [Y_{w,T} \in \cS \mid U_T] \le \\
	& \qquad\qquad\mathrm{e}^{\varepsilon} \bP [Y'_{w, T} \in \cS \mid U'_T] + \delta\ \big|\ |U_T - U'_T|_{\Sigma^{-1}} \le c(\gamma, T)\Bigr] \\
	&\times \bP\bigl[|U_T - U'_T|_{\Sigma^{-1}} \le c(\gamma, T)\bigr].
	\end{align*}
	\ito{Note that it holds
		\begin{align*}
		&\mathbb{P} \Bigl[ \mathbb{P} [Y_{w,T} \in \mathcal{S} \mid U_T] \le \\
		&\qquad\mathrm{e}^{\varepsilon} \mathbb{P} [Y'_{w, T} \in \mathcal{S} \mid U'_T] + \delta\ \big|\ |U_T - U'_T|_{\Sigma^{-1}} \le c(\gamma, T)\Bigr]=1,
		\end{align*}
		since the mechanism satisfies 
		\[\mathbb{P} [Y_{w,T} \in \mathcal{S} \mid U_T] \le 
		\mathrm{e}^{\varepsilon} \mathbb{P} [Y'_{w, T} \in \mathcal{S} \mid U'_T] + \delta
		\]whenever $|U_T - U'_T|_{\Sigma^{-1}} \le c(\gamma, T)$ (i.e. $(U_T, U'_T) \in \mathrm{Adj}_K$). Next, by definition of $c(\gamma, T)$, we have
		\[
		\mathbb{P}\bigl[|U_T - U'_T|_{\Sigma^{-1}} \le c(\gamma, T)\bigr] = \gamma .
		\]
		Consequently, we obtain the desired result.}
\end{proof}

It should be emphasized that such a simple relation is obtained since the prior is Gaussian and the system is linear. 
\subsection{Asymptotic analysis}\label{subsec:asymptotic}
For the conventional DP, it is known that when the system~\eqref{sys} is asymptotically stable, one can design a Gaussian noise which makes the induced mechanism differentially private for any time horizon $ T $~\cite[Corollary~2.9]{kawanoDesignPrivacyPreservingDynamic2020}.
This is because, for an asymptotically stable system, the incremental gain from $ |U_T - U'_T| $ to $ |Y_T - Y'_T| $ is bounded by its $ H^\infty $-norm for any $ T $, and by the definition of DP, the distance $ |U_T - U'_T| $ is also bounded by a predetermined threshold.
That is, even when the horizon of the data to be protected becomes longer, the distance between data sets where their indistinguishability is guaranteed is fixed.

On the other hand, for the proposed BDP, as $ T $ becomes larger, $ |U_T - U'_T| $ tends to take larger values according to the prior $ \bP_{U_T} $.
Consequently, to achieve BDP for a large time horizon $ T $, large noise is required. To see this from Theorem~\ref{thm:BDP_dynamic}, $ c(\gamma, T) > 0 $ is plotted in Fig.~\ref{fig:c_T} as a function of $ T $. As can be seen, as $ T $ increases, $ c(\gamma, T) $ becomes large, and therefore, from \eqref{eq:cond_BDP}, the scale parameter $ \Sigma_w $ of the noise is required to be large to guarantee BDP.
This fact suggests that the privacy requirement of BDP (with fixed $\varepsilon,\gamma,\delta$) for the long (possibly infinite) horizon case is too strong. This issue will be resolved by an appropriate scaling with $T$ to quantify the long-time average privacy.
\begin{figure}
	\centering
	\includegraphics[height=1.8in]{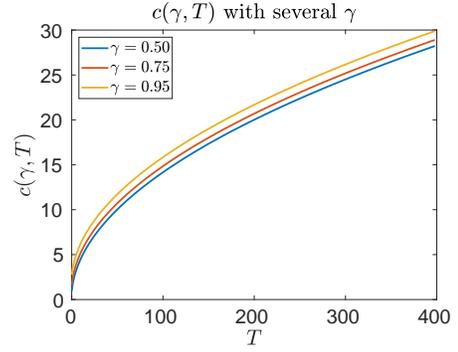}
	\caption{Graph of $c(\gamma,T)$.}
	\label{fig:c_T}
\end{figure}
\section{Design of mechanism}\label{sec:modeling_mechanism}
To motivate additional analysis in this section, let us consider a feedback interconnection of the plant $\cal P$ and the controller $\cal C$:
\begin{align*}
&{\cal P} :\left\{\begin{array}{l}
x_p(t+1) = A_p x_p(t) + B_p u_p(t),\\
y_p(t) = C_p x_p(t), \\
\end{array}\right.\\
&{\cal C} : \left\{\begin{array}{l}
x_c(t+1) = A_c x_c(t) + B_c e(t), \\
u_p(t) = C_c x_c(t), 
\end{array}\right.
\end{align*}
where the control objective is to make the tracking error $e(t):=r(t)-y_p(t)$ small, and its private information is the reference signal $ r(t) $. The attacker, who can access to the output $y_p(t)$, attempts to estimate $r(t)$.   
To prevent this inference, we add noise $v$ to $r$, which leads to the following closed loop dynamics:
\begin{align}\label{eq:interconnected_system}
&\left\{\begin{array}{l}
\bar{x}(t+1)
= \bar{A} \bar{x}(t)
+ \bar{B} ( r(t) + v (t) ),\\
y_p(t) = \bar{C} \bar{x}(t),\\
e(t) = r(t) - \bar{C} \bar{x}(t), \\
\end{array}\right.
\end{align}
where $ \bar{x} (t) := [x_p (t); x_c (t)] $ and
\begin{align*}
&\bar{A} := \begin{bmatrix}
A_p & B_p C_c \\
\sugiura{-B_c C_p} & A_c
\end{bmatrix}, \
\bar{B} := \begin{bmatrix}
0 \\ \sugiura{B_c}
\end{bmatrix}, \ \bar{C} := \begin{bmatrix}
C_p & 0
\end{bmatrix}.
\end{align*}
Suppose that the distribution of \sugiura{$V_T$} is given by ${\cal N}(0,\Sigma_v)$. Then, larger noise $v$ fluctuates $e$ more so that the variance of $E_T := [e(0);\cdots;e(T)]$ is given by
\begin{align}\label{eq:error_variance}
\Theta_T \Sigma_v \Theta_T^\top
\end{align}
with $\Theta_T := {\cal T}_T(\bar A,\bar B, -\bar C, 0)$.

\color{black}
\subsection{Minimization of the noise}

The expression \eqref{eq:error_variance} motivates us to seek the Gaussian noise with the minimum energy among those satisfying the sufficient condition \eqref{eq:cond_BDP} for Bayesian differential privacy derived by Theorem \ref{thm:BDP_dynamic}.
More specifically, we consider the following optimization problem with the covariance matrix of Gaussian noise $\Sigma_w \succ 0$ as the decision variable.
\begin{prob}\label{prob:opt_out-sigma_wadj}
	\begin{align}
	&\underset{\Sigma_w \succ 0}{\text{minimize}} &&{\rm Tr}(\Sigma_w) \label{opt_trSigma}\\
	&\text{subject to} && \lambda_{\max}(\Sigma^{1/2}N_T^\top\Sigma_w^{-1}N_T\Sigma^{1/2}) \le \frac{1}{c(\gamma, T)^2\cR(\varepsilon, \delta)^2}. \label{constraint:modified}
	\end{align}
\end{prob}
The constraint \eqref{constraint:modified} is an \sugiura{inequality} that is equivalent to the \sugiura{inequality} \eqref{eq:cond_BDP}.
Under certain assumptions, this solution can be obtained as follows\sugiura{.}
\begin{thm}\label{thm:min_noise}
	Assume that $N_T$ is full row rank. The optimal solution to problem \ref{prob:opt_out-sigma_wadj} is $\Sigma_w^* \coloneqq c(\gamma, T)^2\cR(\varepsilon, \delta)^2 N_T \Sigma N_T^\top$.
\end{thm}

\begin{proof}
	Denote ${\sf N}:=c(\gamma, T)\cR(\varepsilon, \delta) N_T \Sigma^{1/2}$ so that $\Sigma_w^* = {\sf N} {\sf N}^\top$. Then, \eqref{constraint:modified} is equivalent to $I - {\sf N}^\top \Sigma_w^{-1} {\sf N} \succeq 0$. By the Schur complement,
	\begin{align}
	\begin{bmatrix}
	\Sigma_w & {\sf N} \\ {\sf N}^\top & I
	\end{bmatrix}\succeq 0. 
	\end{align}
	This implies $\Sigma_w \succeq \Sigma_w^*$, and consequently ${\rm Tr}(\Sigma_w)\ge 
	{\rm Tr}(\Sigma_w^*)$. 
\end{proof}

The obtained optimal solution $\Sigma_w^*$ is a constant multiple of the covariance matrix of the distribution of the output data $Y_T = N_T U_T$ when the covariance matrix of the input data $U_T$ is $\Sigma$. This means that it is possible to efficiently conceal the input data from the output data by applying the noise having the same statistics (up to scaling) as the observed data.

\subsection{Input noise mechanism}
\begin{figure}[!t]
	\centering
	\includegraphics[height=0.6in]{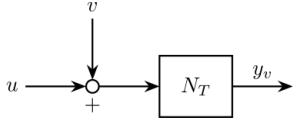}
	\caption{Mechanism with input noise.}
	\label{fig:diagram_in_noise}
\end{figure}
In this subsection, we study the case where noise is added to the input channel; see Fig.~\ref{fig:diagram_in_noise}.
Consider the following system with input noise:
\begin{align}
\left\{\begin{array}{l}
x(t+1) = A x(t) + B (u(t) + v(t)), \\
y_v(t) = C x(t) + D (u(t) + v(t)).
\end{array}\right. \label{sys_input_noise}
\end{align}
As in the aforementioned section, we assume $x(0) = 0$.
The output sequence $Y_{v, T} = [y_v(0);\cdots;y_v(T)]$ can be described as
\begin{align}
Y_{v, T} = N_T U_T + N_T V_T. \label{In_noise}
\end{align}
For the system in \eqref{sys}, adding noise $V_T$ to the input channel is equivalent to adding noise $W_T = N_T V_T$ to the output channel.
For simplicity, we assume that $N_T$ is square and nonsingular; this can be relaxed as in \cite[Corollary 2.16]{kawanoDesignPrivacyPreservingDynamic2020}.
From Theorem \ref{thm:BDP_dynamic}, we obtain the following corollary.
\begin{cor}\label{cor:BDP_dynamic_input_noise}
	Suppose that the prior distribution $\bP_{U_T}$ of $U_T$ is $\cN_{(T+1)m}(0,\Sigma)$.
	The Gaussian mechanism \eqref{In_noise} induced by $V_T \sim \cN_{(T+1)q}(\mu_v,\Sigma_v)$ is $(\bP_{U_T}, \gamma, \varepsilon, \delta)$-Bayesian differentially private at a finite time $T \in \bZ_+$ with $1\ge \gamma \ge 0$, $\varepsilon > 0$, and $1/2 > \delta > 0$, if the covariance matrix $\Sigma_v \succ 0$ is chosen such that
	\begin{align}
	\lambda_{\min}^{1/2}(\Sigma^{-1/2}\Sigma_v\Sigma^{-1/2}) \ge c(\gamma, T)\cR(\varepsilon, \delta) \label{eq:cond_BDP_in}
	\end{align}
	with $c(\gamma, T)>0$ defined in Theorem \ref{thm:BDP_dynamic}.
\end{cor}
\begin{proof}
	The desired result is a straightforward consequence of Theorem \ref{thm:BDP_dynamic}. 
\end{proof}

In \cite[Corollary 2.16]{kawanoDesignPrivacyPreservingDynamic2020}, a sufficient condition for $(\varepsilon,\delta)$-differential privacy in the sense of Definition \ref{def:DP_dynamic_cadj} is given. The result concludes that \emph{the differential privacy level for the input noise mechanism does not depend on the system itself}.
Similarly, \eqref{eq:cond_BDP_in} does not depend on the system matrices in \eqref{sys} either. The difference for the Bayesian case is that \eqref{eq:cond_BDP_in} depends on the covariance $\Sigma$ of the prior distribution of the signals to be protected.

\begin{rem}\label{rem:min_input_noise}
	It is clear from Corollary \ref{cor:BDP_dynamic_input_noise} and Theorem \ref{thm:min_noise} that the minimum energy Gaussian noise that satisfies the sufficient condition for privacy guarantee \eqref{eq:cond_BDP_in} for the input noise mechanism can be easily obtained by 
	\begin{align}
	\Sigma^*_v = c(\gamma, T)^2\cR(\varepsilon, \delta)^2 \Sigma.
	\end{align}
	This characterization allows the natural interpretation that large noise is needed to protect large inputs; see also the next section.
\end{rem}

\section{Numerical example}
\label{sec:example}
Consider the feedback system in Fig.~\ref{fig:feedback}, where 
the plant and controller in \eqref{eq:interconnected_system} are given by 
\begin{align}
A_p &= \left[
\begin{array}{cc}
1.2 & -0.5 \\
1 & 0
\end{array}
\right] ,\ B_p = [-0.3, 0]^\top,\ C_p= [0.2, 0], \label{plant_para} \\
A_c &= \left[
\begin{array}{cc}
1 & 1 \\
0 & 0.1
\end{array}
\right] ,\ B_c = [0, \sugiura{-1}]^\top,\ C_c = [\sugiura{1.5}, 0]. \label{cont_para}
\end{align}
The integral property of the controller enhances the low-frequency tracking performance. The Bode gain diagram is shown in Fig.~\ref{fig:Bode}. Suppose that the reference $r$ is the signals to be protected, and it is public information that its spectrum is concentrated over the frequency range \sugiura{below $3\times10^{-2} \ {\rm rad/s}$}.
To represent this prior information, we took the frequency model for $r$ as in Example \ref{example:freq_model}, which is set to be a lowpass filter (generated by \texttt{lowpass(xi, \sugiura{3e-2})} in MATLAB). Recall that $U_T \sim \cN(0, \Sigma_U)$ with \eqref{eq:freq_covar} and \eqref{eq:freq_covar2}.

\begin{figure}[!t]
	\centering
	\includegraphics[height=0.6in]{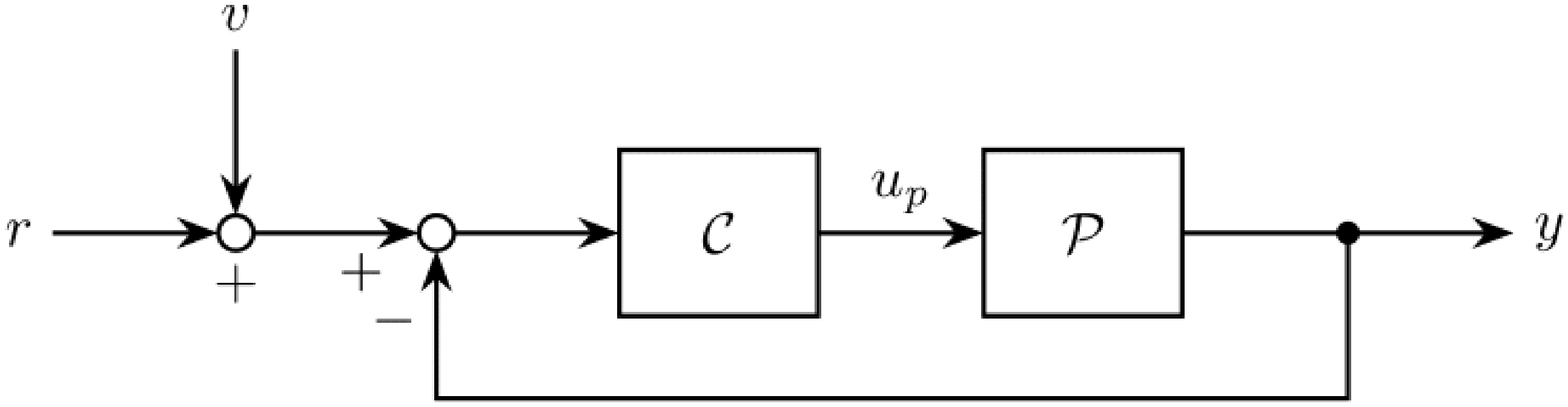}
	\caption{Feedback system with input noise mechanism}
	\label{fig:feedback}
\end{figure}
\begin{figure}
	\centering
	\includegraphics[height=2in]{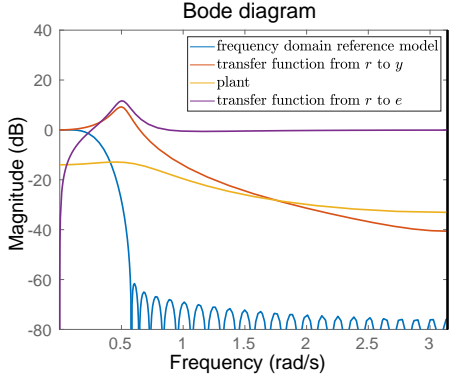}
	\caption{Bode gain diagram for the frequency domain reference model $G_r$ and the transfer functions for the feedback system \eqref{eq:interconnected_system} with \eqref{plant_para}, \eqref{cont_para}.}
	\label{fig:Bode}
\end{figure}
\begin{figure}
	\centering
	\includegraphics[height=1.9in]{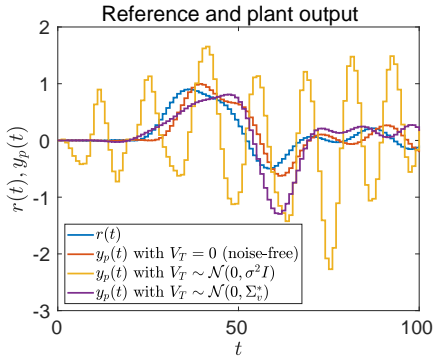}
	\caption{Reference signal $r(t)$ and plant output $y_p(t)$ for three mechanisms.}
	\label{fig:ref_and_output}
\end{figure}
We design input noise to make the system Bayesian differentially private for $\gamma = 0.5,\ T = 100,\ \varepsilon = 100,\ \delta = 0.1$. This leads to 
\begin{equation} 
c(\gamma, T) = 14.1657,\ \cR(\varepsilon, \delta) = 0.0774.
\end{equation}
In what follows we compare the following three cases:
\begin{itemize}
	\item noise-free,
	\item i.i.d. noise: $\sugiura{V}_T \sim \cN_{(T+1)m}(0, \sigma_{\rm iid}^2 I)$ with 
	\begin{align}
	\sugiura{{\rm Tr}(\sigma_{\rm iid}^2 I)} &= \sugiura{c(\gamma, T)^2 \cR(\varepsilon, \delta)^2 \lambda_{\max}(\Sigma_U)(T+1)m} \nonumber\\
	&= \sugiura{121.604},
	\end{align}
	where $\sigma_{\rm iid}$ is the minimum $\sigma$ satisfying \eqref{eq:cond_BDP_in}, and 
	\item the minimum noise obtained in Theorem \ref{thm:min_noise}: $\sugiura{V}_T \sim \cN_{(T+1)m}(0, \Sigma_v^*$), $\Sigma_v^* := c(\gamma, T)^2\cR(\varepsilon, \delta)^2\Sigma_U$ with
	\begin{align}
	\sugiura{{\rm Tr}(\Sigma_v^*) = 8.2574}.
	\end{align}
\end{itemize}

\sugiura{Fig.~\ref{fig:ref_and_output} shows the reference signal $r(t)$ and plant output $y_p(t)$ for these three cases.}
It can be seen that the output error for the noise-free case is the smallest. This is because the (realized) trajectory of $r$ is fully utilized allowing for some possibility that information about $r$ may leak from $y_p$. On the other hand, the other two cases guarantee the same level of Bayesian differential privacy. Note that the error fluctuation is suppressed in the minimum noise case.
Statistically, the output fluctuation caused by the added noise \red{$v$} can be evaluated by \eqref{eq:error_variance}. The value ${\rm Tr}(\Theta_T \Sigma_v \Theta_T^\top ) / (c(\gamma, T) \cR(\varepsilon, \delta))^2$
is given by
$
{\rm Tr}(\Theta_T \Sigma_U \Theta_T^\top ) = \sugiura{8.1998}
$
for the minimum noise case, which is smaller than 
$    \lambda_{\max}(\Sigma_U){\rm Tr}(\Theta_T  \Theta_T^\top )=\sugiura{55.4202}
$ for the i.i.d. case.

The interpretation is as follows: The i.i.d.~noise has uniform frequency distribution, which implies it adds more out-of-band noise than the minimum one. However, this out-of-band component does not contribute to the protection of $r$ since it is easily distinguished from $r$ thanks to its prior information in Fig.~\ref{fig:Bode}. Nevertheless, this out-of-band noise largely degrades the tracking performance. 

\begin{rem}
	The out-of-band noise as in the i.i.d. case is effective when the prior distribution of the signals to be protected is not public information. That is, this noise can prevent the attacker from inferring the prior distribution e.g., via the empirical Bayes. 
\end{rem}

\begin{rem}
	\red{Lastly, we would like to note that for the Gaussian prior $ \mathcal{N}(0,\Sigma_U) $ in the numerical example, a reference signal $ U'_T $ that deviates significantly from mean $ 0 $ in the sense of $ |0 - U'_T|_{\Sigma_U^{-1}} $ can be seen as an outlier. Therefore, out-of-band signals having large values $ |U'_T|_{\Sigma_U^{-1}} $ can be regarded as outliers.
		Bayesian differentially private mechanism provides privacy guarantees not only for in-band signals but also for out-of-band ones.
		In particular, the parameter $ \gamma $ for Bayesian differential privacy determines the extent to which privacy guarantees can be provided for out-of-band signals.}
\end{rem}

\section{Conclusion}\label{sec:conclusion}
In this study, we introduced Bayesian differential privacy for linear dynamical systems using prior distributions of input data to provide privacy guarantees 
even for input data pairs with large differences, and gave sufficient conditions to achieve it. Furthermore, we derived the minimum energy Gaussian noise that satisfies the condition.
As noticed in Subsection~\ref{subsec:asymptotic}, any finite noise cannot guarantee the Bayesian differential privacy for the infinite horizon case. This issue will be addressed in future work.

\section*{Acknowledgment}

This work was supported in part by JSPS KAKENHI under Grant Number \sugiura{JP21H04875}.

\ifCLASSOPTIONcaptionsoff
  \newpage
\fi



%



\bibliography{bare_jrnl}
\bibliographystyle{IEEEtran}




\end{document}